\documentclass[12pt]{amsart}

\usepackage[pdftex]{graphicx}

\usepackage[T2A]{fontenc}
\usepackage[utf8]{inputenc}
\usepackage[english]{babel}

\usepackage{amsmath}
\usepackage{amssymb}
\usepackage{amscd}
\usepackage{amsfonts,amsmath,amsthm}

\usepackage[matrix,arrow,curve]{xy}
\usepackage{xfrac}
\usepackage{color}

\newtheorem{theorem}{Theorem}
\newtheorem{prop}{Proposition}

\newtheorem{lemma}{Lemma}
\newtheorem*{main lemma}{The main lemma}
\newtheorem*{remark}{Remark}
\newtheorem{definition}{Definition}
\newtheorem{dfn}{Definition}

\title[Each simplex in $\mathbb{R}^4$ has a point with   $14$ emanating normals]{Each simplex in $\mathbb{R}^4$ has  a point with at least $14$ emanating normals to its boundary}

\author{Ivan Nasonov, Gaiane Panina}

\address{I. Nasonov: St. Petersburg State University   wanua-nasonov-i04@yandex.ru; G. Panina: St. Petersburg department of Steklov institute of mathematics
 gaiane-panina@rambler.ru  }
\date{\today}

\begin{document}

\begin{abstract}

We prove that each simplex in $\mathbb{R}^4$   contains a point  with at least $14$ emanating normals to the boundary. 

Earlier, the authors proved that each generic simple polytope in $\mathbb{R}^4$   contains a point  with at least $12$ emanating normals to the boundary. 

This problem belongs to a piecewise linear counterpart of a longstanding problem concerning normals to smooth convex bodies.
\end{abstract}

\maketitle

\section{Introduction}

\medskip

We continue studying concurrent normals problem for convex polytopes.

\begin{dfn}
  Let $\mathbf{P}\subset \mathbb{R}^4$  be a compact convex polytope with non-empty interior, $y\in Int~\mathbf{P}$.
 A \textit{normal to the boundary of $\mathbf{P}$ emanating from }$y$ is a line $yz$ such that $z\in \partial \mathbf{P}$   such that  $yz$ is orthogonal to some support plane containing $z$. 
\end{dfn}

It is known that:\begin{itemize}
                   \item Each generic  polytope in $\mathbb{R}^3$ has a point in its interior with at least $10$ normals to the boundary, 
see \cite{NasPan} and  a slightly different proof in \cite{Angst}. 
\item There exists a simplex in $\mathbb{R}^3$ whose maximal number of emanating normals is equal to $10$.
                   \item Each generic 
    polytope  $\mathbf{P} \subset \mathbb{R}^4$   has a point with (at least) $12$ emanating normals to the boundary  \cite{NasPan2}.
                   \item  Moreover, each generic 
    polytope  $\mathbf{P} \subset \mathbb{R}^n$   has a point with (at least) $2n+4$ emanating normals to the boundary \cite{NasPan2}.
                 \end{itemize}

Our main result is:

\begin{theorem}\label{Thmmain}
  Each generic \footnote{The definition of genericity is given in Section \ref{SecGen}} simplex in $\mathbb{R}^4$ contains a point with at least $14$ emanating normals to the boundary.
\end{theorem}

 (In fact, a stronger fact is proven in Proposition \ref{propInc}.)
 
 \medskip
 
 This bound is exact: according to personal communications with Chiara Meroni, there exists a simplex in $\mathbb{R}^4$ with no point having more than $14$ emanating normals.

 \medskip
 
 The paper relies very much on ideas developed in \cite{NasPan} and \cite{NasPanSiersma}. To make the paper self-contained, we collect the constructions and facts in Appendix. We refer the reader to these papers for more detailed explanations of the backgrounds and also for more figures.


\section{Proof of the theorem}\label{SecProof}

Counting normals emanating from a point $y$ amounts to counting critical points of the \textit{squared distance function} $$SQD_y(x)=|x-y|^2:\partial \mathbf{P}\rightarrow \mathbb{R}.$$  Generally, it can be described as Morse function, which gives us some useful tools, see Section \ref{SecActReg} for details.

\bigskip

The key idea of the proof is to use the centers of inscribed spheres. 

\begin{definition}
A sphere is said to be inscribed in a convex polytope $\mathbf{P}\subset \mathbb{R}^n$ if it is contained in $\mathbf{P}$, and the sphere is tangential to (at least) $n+1$ facets of $P$.

A spherical polytope is a polytope that fits in some open hemisphere. It is always assumed to be spherically convex. For a spherical polytope $P\subset S^n$, an inscribed sphere is defined in the same way.
\end{definition}

Simple observations are:\begin{enumerate}
                          \item An inscribed sphere exists for every polytope.
                          \item An inscribed sphere may not be unique. However, for a simplex it is unique.
                          \item For spherical polytopes, the sphere of maximal volume contained in $P$ is not necessarily inscribed.
                        \end{enumerate}
                        
                        \bigskip

For a convex polytope $\mathbf{P}\subset \mathbb{R}^n$ and its vertex $V$ we define a spherical polytope $Spher(V) \subset S^{n-1}$ as an intersection of $\mathbf{P}$ with small sphere, centred in $V$. We assume that this $S^{n-1}$ is equipped with  metric of the standard unit sphere.

Remind that for a convex spherical polytope $P$ (or a polygon) its \textit{core} is defined as $$Core(P) = \{y\in P: |Vy|<\pi/2 \text{ for every vertex $V$ of $P$}\}.$$

\begin{definition}

Let $\Delta\subset S^3$ be a convex spherical tetrahedron. Let $O_{\Delta}$ be the center of the inscribed sphere. The polytope $\Delta$ is called \textit{perfect} if $O_{\Delta}\in Core(\Delta)$.

A vertex $V$ of a convex polytope $\mathbf{P}\subset \mathbb{R}^4$ is called \textit{perfect} if $Spher(V)$ is a perfect  spherical tetrahedron.
\end{definition}
  \bigskip
  
We shall make use of the spherical analog of the squared distance function for spherical polytopes.

Let $P$ be a spherical polytope with $y\in P$. The function $$SQD^{sph}_y(x) =|xy|^2:\partial P\rightarrow \mathbb{R}$$ is called the \textit{spherical squared distance function}. It  also behaves like a Morse function for general $y\in P$, but in comparison with the Euclidean $SQD$, there are new phenomena. For example, maxima are not necessarily attained at vertices. 

  \bigskip
  
\begin{definition}
    A spherical segment $AB\subset S^{n-1}$ is called \textit{short} if $|AB|<\pi/2$; otherwise, it is called \textit{long}.

A critical point of $SQD_y^{sph}$ is short (respectively, long) if the normal segment $yz$ is short (respectively, long).

\end{definition}
 Note that for $y\in Core(P)$ all the critical points are short.

  \begin{lemma}\label{lemmaSaddle}
For a perfect spherical simplex, an edge $e$ contains a short critical point (which is automatically a saddle point) of $SQD^{sph}_{O_\Delta}$  iff $O_{\Delta}$ projects onto $e$. 
\end{lemma}
\begin{proof} There are two conditions for an edge to contain a saddle point: one is that $y$ projects onto $e$, and the other one
is that  angles between the plane $(ey)$ and the faces incident to $e$ are acute. This is true for the point $y=O_{\Delta}$ since the plane $(ey)$  bisects the angle at the edge $e$.
\end{proof}

 \begin{lemma}\label{lemmaTriangTwoEdges}
Let $F=ABC$ be a spherical triangle, and let $X \in Core(F)$. Then $X$ projects onto at least two edges of $F$. 
\end{lemma}
\begin{proof}

Assume the contrary. Then  the function $SQD^{sph}_X: \partial F\rightarrow \mathbb{R}$ has exactly one short maximum attained, say, at the vertex $A$. The vertices $B$ and $C$ are not maxima. Then there exists an obtuse angle incident to the vertex $B$ and  an obtuse angle incident to the vertex  $C$,  see Fig. \ref{pic2cases}.

\textbf{Case 1, } Fig. \ref{pic2cases}. These two obtuse angles lie in different triangles  $AXB$, $BXC$, $CXA$. Then one of them (say, $AXB$) necessarily has an obtuse angle incident to $X$. The third angle in this triangle is acute. Finally $\angle AXB>\pi/2, \angle ABX > \pi/2, \angle BAX < \pi/2, |XA|<\pi/2$. Such a triangle does not exist.

\textbf{Case 2.} The obtuse angles are $XBC$ and $XCB$. Then we obtain local minima of $SQD_X^{sph}$ at $AB$ and $AC$. 
  \end{proof}
  
  \bigskip
     
\begin{figure}[h]\label{pic2cases}
 \includegraphics[width=0.8\linewidth]{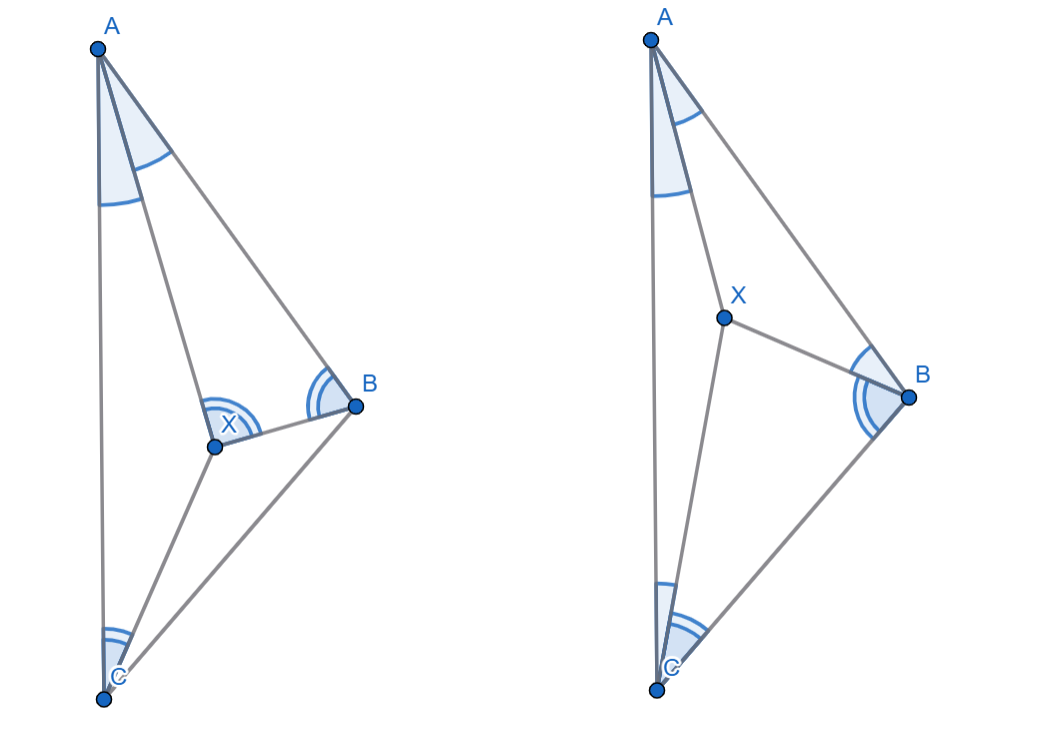}
  \caption {Two cases for Lemma \ref{lemmaTriangTwoEdges}. Acute angles are marked by single arcs, obtuse angles are marked by double.}
  \label{pic2cases}
\end{figure}

  \bigskip
  
\begin{lemma}\label{lemmaInc}
If $\Delta$ is a perfect  spherical  tetrahedron, then the spherical squared distance function $SQD^{sph}_{O_\Delta}$ has at least $10$ short critical points.
\end{lemma}

\begin{proof}
The maxima of $SQD^{sph}_{O_\Delta}$ are attained at some vertices, and all maxima are short. Local minima are attained at some facets and saddle points at some edges, all of them are also short. Prove that $SQD^{sph}_{O_\Delta}$ has at least $4$ saddles which will immediately imply the desired statement. 

Since $O_{\Delta}$ is the center of the inscribed sphere, it projects to all the facets.

By Lemmata \ref{lemmaSaddle}, \ref{lemmaTriangTwoEdges}, each facet contributes at least two critical edges. Each edge  occurs twice, so
 $O_\Delta$ projects onto at least 4 edges of $\Delta$. 
\end{proof}

\begin{prop}\label{propInc}
Suppose that a generic convex polytope $\mathbf{P}\subset \mathbb{R}^4$ has at least one perfect vertex. Then $\mathbf{P}$ contains a point with at least $14$ emanating normals.
\end{prop}

\begin{proof}
First remind that if a point $y\in \mathbf{P}$ lies very close to a vertex $V$, then the critical points of $SQD_y$ lying on the faces that are incident to $V$ (excluding the point $V$) are in bijection with short critical points of $SQD_y^{sph}$.

Next, we make use of a technique developed  in \cite{NasPan}, \cite{NasPan2}.

Pick the ray $VO_\Delta$, where $\Delta = Spher(V)$. Let the point $y$ keep moving along the ray with starting position $y = V$. Just after the start we have (at least) $12$ critical points of $SQD_y$: a local maximum at $V$,  at least $10$ critical points on the faces incident to $V$ (by Lemma \ref{lemmaInc}) and a global maximum. If there are more critical points, their total number is at least $14$.

The point $y$ keeps moving until the \textit{first event},  which is either a bifurcation of $SQD_y$,  or  the point $y$ leaves the
interior of $\mathbf{P}$. We have the following case analysis:

(a) The point $y$ leaves the polytope $\mathbf{P}$. This means that $y$ hits one
of the facets which is not incident to the vertex $V$. Right before this,
$y$, we have  an additional minimum attained at this facet. This minimum is not mentioned among $12$ <<old>> critical points. The <<old>> critical points persist since there was no bifurcation.

(b)The first event is a birth of two new critical points. This gives us $14$ normals.

(c) One of the <<old>> critical points dies. The pairwise distances between $11$  <<old>>  critical points (excluding the global maximum) keep increasing, so  they cannot collide. 
The only remaining case is if some 2-saddle $s$ meets the global maximum $V'$. Recall
that $s$ is the orthogonal projection of $y$ onto one of the edges incident to $V$. Therefore, $V'$ is the endpoint of this edge. However, if $s$
meets $V'$ before $y$ leaves $\mathbf{P}$, then $|Vy|>|VV'|$, so $V'$ can't be a maximum of $SQD_y$ for the point $y=V$.

\end{proof}

Now  Theorem \ref{Thmmain} follows from Proposition \ref{propInc} and the following lemma:

\begin{lemma}
Let $O$ be the center of the inscribed sphere of a simplex $\mathbf{P}$. Then the maximum of $SQD_O$ is attained at a perfect vertex.
\end{lemma}

\begin{proof} Indeed, let $V$ be this maximum and $\Delta = Spher(V)$. The points $O, O_{\Delta}, V$ are collinear, so $O_{\Delta} \in Core(\Delta)$.
\end{proof}

\section{Appendix}\label{SecAppendix}

Most of the definitions and lemmata of the section are borrowed from \cite{NasPanSiersma}  and \cite{NasPan}; we repeat them  briefly for the sake of completeness.
\subsection{Active regions and bifurcation set}\label{SecActReg}
Let $\mathbf{P}\subset \mathbb{R}^4$ be a compact simple convex polytope with non-empty interior.

Let $F$ be a face of  $\mathbf{P}$ (that is, either a facet, a $2$-face, an edge, or a vertex). The \textit{active region} $\mathcal{AR}(F)$ is the set of all points $y\in Int\, \mathbf{P}$ such that $F$ contains the base of some normal emanating from $y$. 

The\textit{ bifurcation set}  $\mathcal{B}(\mathbf{P})$ is the union of the boundaries of active regions. It is the piecewise linear counterpart of the focal set that appears in the smooth case.

\medskip

Although the function $SQD_y$ is  non-smooth,
it can be treated as a Morse function.  So counting critical points  of $SQD_{y}$ amounts to counting normals, exactly as it is in the smooth  case.

\begin{lemma}

For $y \notin \mathcal{B}(\mathbf{P})$, we have:
\begin{enumerate}
  \item Local maxima of $SQD_y$  are attained at some  vertices of $\mathbf{P}$.
  \item Local minima of $SQD_y$ are attained at some facets of $\mathbf{P}$.
  \item Saddles with Morse index $2$ of $SQD_y$ are attained at some  edges of $\mathbf{P}$.
   \item Saddles with Morse index $1$ of $SQD_y$ are attained at some  $2$-faces of $\mathbf{P}$.
   \item If $y$ crosses transversally one sheet of $\mathcal{B}(\mathbf{P})$, a pair of critical points of  $SQD_{y}$ either dies or is born. This situation is called a bifurcation.

\end{enumerate}

\end{lemma}

\textbf{Convention:}  We call local maxima (local minima) just maxima and minima for short.

\medskip

We say that a point $y\in Int(\mathbf{P})$ \textit{projects} to a face $F$  if the orthogonal projection of $y$ to the affine hull of $F$ lies in the interior of the face $F$.

\subsection{Genericity assumptions.}\label{SecGen}

In the paper we restrict ourselves to \textit{generic  polytopes}:

\begin{definition}\label{DefGen} A polytope $\mathbf{P}$ with the vertex set $\mathbf{V}$ is generic if  affine hulls of any two subsets of
$\mathbf{V}$ are neither parallel nor orthogonal, unless this is dictated by unavoidable reasons. For example,
  \begin{description}
  \item[i] No two edges  are orthogonal.
  \item[ii]  No edge is parallel to a face, unless they are incident. 
\end{description}
\end{definition}

Important observations are: (1) {For a generic polytope, the sheets of $\mathcal{B}(\mathbf{P})$ intersect transversally.}
(2) {Any polytope can be turned to a generic one by a close to $id$ projective transform.} (3) If $\mathbf{P}$ is generic, then any sheet of $\mathcal{B}(\mathbf{P})$ 
related to a face contains no vertices of $\mathbf{P}$ except for vertices of the face.

\medskip

\textbf{Acknowledgement.}
The work of Ivan Nasonov was performed at the Saint Petersburg Leonhard Euler International Mathematical Institute and supported by the Ministry of Science and Higher Education of the Russian Federation (agreement no. 075–15–2025–343)

\end{document}